\documentclass[12pt, a4paper]{amsart}
\usepackage{amsmath}
\usepackage{geometry,amsthm,graphics,tabularx,amssymb,
shapepar}
\usepackage{amscd}
\usepackage[all,2cell,dvips]{xy}

\newcommand{\CFH}{{\mathcal {FH}}}

\newcommand{\Ad}{{\mathrm{Ad}}}

\newcommand{\GL}{{\mathrm{GL}}}
\newcommand{\U}{{\mathrm{U}}}

\newcommand{\Hom}{{\mathrm{Hom}}}

\renewcommand{\Re}{{\mathrm{Re}}}

\newcommand{\Sp}{{\mathrm{Sp}}}

\newcommand{\tr}{{\mathrm{tr}}}

\newcommand{\vsp}{{\vspace{0.2in}}}

\newcommand{\con}{\textit{C}}

\newcommand{\oO}{\operatorname{O}}

\newcommand{\oU}{\operatorname{U}}
\newcommand{\oK}{\operatorname{K}}
\newcommand{\oZ}{\operatorname{Z}}

\newcommand{\gC}{{\mathfrak g}_{\C}}

\newcommand{\C}{\mathbb{C}}
\newcommand{\R}{\mathbb R}

\newcommand{\K}{\mathbb{K}}

\newcommand{\bt}{\mathbf{t}}

\newcommand{\abs}[1]{\lvert#1\rvert}
\newcommand{\tbt}{\tilde{\mathbf{t}}}

\newcommand{\la}{\langle}
\newcommand{\ra}{\rangle}

\newcommand{\be}{\begin {equation}}
\newcommand{\ee}{\end {equation}}
\newcommand{\bee}{\begin {equation*}}
\newcommand{\eee}{\end {equation*}}

\theoremstyle{Theorem}
\newtheorem{introconjecture}{Conjecture}

\newtheorem{introtheorem}[introconjecture]{Theorem}

\theoremstyle{Theorem}
\newtheorem{lem}{Lemma}[section]

\newtheorem{leml}[lem]{Lemma}
\newtheorem{prpl}[lem]{Proposition}

\theoremstyle{Theorem}
\newtheorem{prp}{Proposition}[section]

\newtheorem{lemp}[prp]{Lemma}

\newtheorem{prpp}[prp]{Proposition}

\theoremstyle{Plain}

\theoremstyle{Definition}

\begin{document}

\title{Uniqueness of Bessel models: the archimedean case}

\author [D. Jiang] {Dihua Jiang}
\address{School of Mathematics\\
University of Minnesota\\
206 Church St. S.E., Minneapolis\\
MN 55455, USA} \email{dhjiang@math.umn.edu}

\author [B. Sun] {Binyong Sun}
\address{Hua Loo-Keng Key Laboratory of Mathematics, Academy of Mathematics and Systems Science\\
Chinese Academy of Sciences\\
Beijing, 100190,  P.R. China} \email{sun@math.ac.cn}

\author [C.-B. Zhu] {Chen-Bo Zhu}
\address{Department of Mathematics\\
National University of Singapore\\
2 Science drive 2\\
Singapore 117543} \email{matzhucb@nus.edu.sg}



\begin{abstract}
In the archimedean case, we prove uniqueness of Bessel models for
general linear groups, unitary groups and orthogonal groups.
\end{abstract}

\maketitle

\section{Introduction}

Let $G$ be one of the classical Lie groups
\begin{equation}\label{groups}
  \GL_{n}(\R),\,\GL_{n}(\C),\, \oU(p,q),\,\oO(p,q),\, \oO_{n}(\C).
\end{equation}
In order to consider Bessel models for $G$, we consider, for each
non-negative integer $r$ satisfying
\[
n\geq 2r+1,\quad p\geq r,\quad q\geq r+1,
\]
the $r$-th Bessel subgroup
\[
  S_r=N_{S_r}\rtimes G_0
\]
of $G$, which is a semidirect product and which will be described
explicitly in Section \ref{bes}. Here $N_{S_r}$ is the unipotent
radical of $S_r$, and $G_0$ is respectively identified with
\begin{equation}\label{groups0}
  \GL_{n-2r-1}(\R),\,\GL_{n-2r-1}(\C),\,\oU(p-r,q-r-1),\,\oO(p-r,q-r-1),\,
  \oO_{n-2r-1}(\C).
\end{equation}

Let $\chi_{S_r}$ be a generic character of $S_r$ as defined in
Section \ref{gen}. The main result of this paper is the following
theorem, which is usually called the (archimedean) local uniqueness
of Bessel models for $G$.

\begin{introtheorem}\label{main}
Let $G$, $G_0$, $S_r$ and $\chi_{S_r}$ be as above. For every
irreducible representation $\pi$ of $G$ and
$\pi_0$ of $G_0$ both in the class $\CFH$, the inequality
\[
   \dim \Hom_{S_r}(\pi\widehat{\otimes}
   \pi_0,\chi_{S_r})\leq 1
\]
holds.
\end{introtheorem}

We would like to make the following remarks on Theorem \ref{main}.
The symbol ``$\widehat{\otimes}$" stands for the completed
projective tensor product of complete, locally convex topological
vector spaces, and ``$\Hom_{S_r}$" stands for the space of
continuous $S_r$-intertwining maps. Note that $\pi_0$ is viewed as a
representation of $S_r$ with the trivial $N_{S_r}$-action. As is
quite common, we do not distinguish a representation with its
underlying space.

Recall that a representation of $G$ is said to be in the class $\CFH$ if it is Fr\'{e}chet, smooth, of moderate
growth, admissible and $\oZ(\gC)$-finite. Here and as usual,
$\oZ(\gC)$ is the center of the universal enveloping algebra
$\oU(\gC)$ of the complexified Lie algebra $\gC$ of $G$. Of course,
the notion of representations in the class $\CFH$ fits all real
reductive groups. The interested reader may consult \cite{Cass} and
\cite[Chapter 11]{W2} for more details.

If $r=0$, then $S_r=G_0$, and Theorem \ref{main} is the multiplicity
one theorem proved by Sun and Zhu in \cite{SZ} (and independently by
Aizenbud and Gourevitch in \cite{AG} for general linear groups). If
$G_0$ is the trivial group, then Theorem \ref{main} asserts
uniqueness of Whittaker models for $\GL_{2r+1}(\R)$,
$\GL_{2r+1}(\C)$, $\oU(r,r+1)$, $\oO(r,r+1)$ and $\oO_{2r+1}(\C)$.
See \cite{Shl74}, \cite{CHM} for local uniqueness of Whittaker
models for quasi-split groups (or \cite{JSZ} for a quick proof).
Hence the family of Bessel models interpolates between the Whittaker
model ($G_0$ is trivial) and the spherical model ($r=0$).

It is a basic problem in representation theory to establish various
models with good properties. In particular, this has important
applications to the classification of representations and to the
theory of automorphic representations.

Whittaker models for representations of quasi-split reductive groups
over complex, real and p-adic fields and their local uniqueness
property are essential to the Langlands-Shahidi method (\cite{Sh88})
and the Rankin-Selberg method (\cite{B05}) to establish the
Langlands conjecture on analytic properties of automorphic
$L$-functions (\cite{GS88}).

The notion of Bessel models originates from classical Bessel
functions and it was first introduced by Novodvorski and
Piatetski-Shapiro (\cite{NPS73}) to study automorphic $L$-functions for
$\Sp(4)$. For orthogonal groups, the Bessel models are
essential to establish analytic properties of automorphic
$L$-functions as considered in \cite{GPSR97}. The analogue for
unitary groups is expected (see \cite{BAS}, for example). More
recently, Bessel models are used in the construction of automorphic
descents from the general linear groups to certain classical groups
(\cite{GRS99}), as well as in the construction of local
descents for supercuspidal representations of p-adic groups
(\cite{JS03}, \cite{S08}, \cite{JNQ08}, and \cite{JNQ09}). Further
applications of Bessel models to the theory of automorphic forms and
automorphic $L$-functions are expected.

We remark that the local uniqueness of the Bessel models is one of
the key properties, which makes applications of these models
possible. An important purpose of this paper is to show that the archimedean
local uniqueness of general Bessel models can be reduced to the uniqueness of the
spherical models proved in \cite{SZ} (i.e. $r=0$ case).  The key idea in this reduction is to construct an
integral $I_{\mu}$ (Equation \eqref{integral} in Section
\ref{defint}), where $\mu$ is a (non-zero) Bessel functional.
We note that for p-adic fields, the reduction to the p-adic spherical models (proved in
\cite{AGRS}) is known by the work of Gan, Gross and Prasad (\cite{GGP08}).  The approach of this paper works for the p-adic
local fields as well.

\vsp We now describe the contents and the organization of this
paper. In Section \ref{pre}, we recall the general set-up of the
Bessel models. In Section \ref{str}, we outline our strategy, and give the proof
of Theorem A, based on two propositions on the aforementioned integral $I_{\mu}$ (Propositions
\ref{nonzero} and \ref{converge}). This integral depends on a complex parameter $s$. Proposition
\ref{nonzero} states that $I_{\mu}$, when evaluated at a certain
point of the domain, is absolutely convergent and nonzero. On the
other hand, Proposition \ref{converge} asserts that $I_{\mu}$
converges absolutely for all points of the domain when the
real part of the parameter $s$ is large, and it defines a
$G$-invariant continuous linear functional on a
representation of $G'\times G$ in the class $\CFH$, where $G'\supset  G$ is one
of the spherical pairs considered in \cite{SZ}.
The proof of Proposition \ref{nonzero} and Proposition \ref{converge} are
given in Sections \ref{pnonzero} and \ref{conv2}, respectively.
Section \ref{int} is devoted to an explicit integral formula
(Proposition \ref{intf}), as a preparation for Section \ref{conv2}.

\vsp \noindent Acknowledgements: the authors would like to thank D.
Barbasch, J. Cogdell, D. Soudry, R. Stanton, and D. Vogan for
helpful conversation and communication. Dihua Jiang is supported in
part by NSF grant DMS-0653742 and by the Chinese Academy of
Sciences. Binyong Sun is supported by NSFC grants 10801126 and
10931006. Chen-Bo Zhu is supported in part by NUS-MOE grant
R-146-000-102-112.

\section{Bessel subgroups and generic characters}\label{pre}

\subsection{Bessel subgroups}\label{bes}

In order to describe the Bessel subgroups uniformly in all five
cases, we introduce the following notations. Let $\K$ be a
$\R$-algebra, equipped with an involution $\tau$. In this article,
$(\K,\tau)$ is assumed to be one of the pairs
\begin{equation}\label{five}
    (\R\times \R,\tau_\R),
   \,\,(\C\times  \C,\tau_\C), \,\, (\C, \overline{\phantom{a}}\,),\,\,(\R, 1_\R),
   \, \,(\C,1_\C),
\end{equation}
where $\tau_\R$ and $\tau_\C$ are the maps which interchange the
coordinates, ``$\overline{\phantom{a}}$" is the complex conjugation,
$1_\R$ and $1_\C$ are the identity maps.

Let $E$ be a hermitian $\K$-module, namely it is a free $\K$-module
of finite rank, equipped with a non-degenerate $\R$-bilinear map
\[
  \la\,,\,\ra_E:E\times E\rightarrow \K
\]
satisfying
\[
     \la u,v\ra_E=\la v,u\ra_E^\tau, \quad \la au,v\ra_E=a\la u,
     v\ra_E,\quad a\in \K,\, u,v\in E.
\]
Denote by $G:=\oU(E)$ the group of all $\K$-module automorphisms of
$E$ which preserve the form $\la\,,\,\ra_E$.

Assume that $E$ is nonzero. Let $r\geq 0$ and
\[
   0=X_0\subset X_1\subset\cdots\subset X_r\subset X_{r+1}
\]
be a flag of $E$ such that
\begin{itemize}
  \item
    $X_i$ is a free $\K$-submodule of $E$ of rank $i$, $i=0,1,\cdots, r, r+1$,
  \item
    $X_r$ is totally isotropic, and
  \item
    $X_{r+1}=X_r\oplus \K v_0'$ (orthogonal direct sum), with  $v_0'$ a non-isotropic vector.
\end{itemize}

A group of the form
\begin{equation}
\label{defBessel}
  S_r:=\{x\in G\mid (x-1)X_{i+1}\subset
  X_i,\,i=0,1,\cdots,r\}
\end{equation}
is called a $r$-th Bessel subgroup of $G$.

\vsp To be more explicit, we fix a totally isotropic free $\K$-submodule $Y_r$ of
\[
  v_0'^\perp:=\{v\in E\mid \la v, v_0'\ra_E=0\}
\]
of rank $r$ so that the pairing
\[
  \la \,,\,\ra_E: X_r\times Y_r\rightarrow \K
\]
is non-degenerate. Write
\[
  E_0:=v_0'^\perp \cap (X_r\oplus Y_r)^\perp.
\]
Then $E$ is decomposed into an orthogonal sum of three submodules:
\begin{equation}\label{Eoriginal}
   E=(X_r\oplus Y_r)\oplus E_0\oplus \K v_0'.
\end{equation}

According to the five cases of $(\K,\tau)$ in (\ref{five}), $G$ is
one of the groups in (\ref{groups}). By scaling the form
$\la\,,\,\ra_E$, we assume that
\[
\la v_0', v_0'\ra_E=-1,
\]
then $G_0:=\oU(E_0)$ is one of the groups in (\ref{groups0}). The
Bessel subgroup $S_r$ is then a semidirect product
\begin{equation}\label{bg}
  S_r=N_{S_r}\rtimes G_0,
\end{equation}
where $N_{S_r}$ is the unipotent radical of $S_r$.

\subsection{Generic characters}\label{gen}

Write
\[
  L_i:=\Hom_\K (X_{i+1}/X_i, X_i/X_{i-1}), \quad i=1,2,\cdots,r,
\]
which is a free $\K$-module of rank $1$. For any $x\in S_r$, $x-1$
obviously induces an element of $L_i$, which is denoted by
$[x-1]_i$. Denote by $[x]_0$ the projection of $x$ to $G_0$. It is
elementary to check that the map
\begin{equation}
\label{deta}
   \begin{array}{rcl}
  \eta_r: S_r&\rightarrow &C_r:=G_0\times L_1\times L_2\times \cdots \times
  L_r,\\
  x&\mapsto& ([x]_0, [x-1]_1, [x-1]_2, \cdots, [x-1]_r)
  \end{array}
\end{equation}
is a surjective homomorphism, and every character on $S_r$ descends
to one on $C_r$. A character on $S_r$ is said to be generic if its
descent to $C_r$ has nontrivial restriction to every nonzero
$\K$-submodule of $L_i$, $i=1,2,\cdots, r$.

\section{The strategy, and proof of Theorem \ref{main}}\label{str}

\subsection{The group $G'$}\label{Gprime}

Introduce
\[
 E':=E\oplus \K v',
\]
with $v'$ a free generator. View it as a hermitian $\K$-module under
the form $\la\,,\,\ra_{E'}$ so that
\[
  \la\,,\,\ra_{E'}|_{E\times E}=\la\,,\,\ra_E,\quad \la
  E,v'\ra_{E'}=0\quad \textrm{and} \quad\la v',v'\ra_{E'}=1.
\]
Then $E'$ is the orthogonal sum of two submodules:
\[
 E'=(X_{r+1}'\oplus Y_{r+1}')\oplus E_0,
\]
where
\[
   X_{r+1}':=X_r\oplus \K(v_0'+v')\quad \textrm{and}\quad Y_{r+1}':=Y_r\oplus
   \K (v_0'-v')
\]
are totally isotropic submodules.

Write $G':=\oU(E')$, which contains $G$ as the subgroup fixing $v'$.
Denote by $P'_{r+1}$ the parabolic subgroup of $G'$ preserving
$X'_{r+1}$, and by $P_r$ the parabolic subgroup of $G$ preserving
$X_r$. As usual, we have
\begin{equation}
\label{dbigP}
   P_{r+1}'= N_{P'_{r+1}}\rtimes (G_0 \times \GL_{r+1}) \subset G'
   \quad\textrm{and}
   \end{equation}
\begin{equation}
\label{dsmallP}
   \quad P_r=N_{P_r}\rtimes (G_0'\times \GL_r)\subset G,
\end{equation}
where $N_{P'_{r+1}}$ and $N_{P_r}$ are the unipotent radicals of
$P'_{r+1}$ and $P_r$, respectively,
\[
  \GL_{r+1}:=\GL_\K(X_{r+1}')\supset \GL_r:=\GL_{\K}(X_r),
\]
and
\[
  G_0':=\oU(E_0')\supset G_0, \quad \textrm{with } E_0':=E_0\oplus \K v_0'.
\]
Write
\[
 N_{r+1}=\{x\in \GL_{r+1}\mid (x-1)X'_{r+1}\subset X_r,\,(x-1)X_i\subset
  X_{i-1},\,i=1,2,\cdots, r\},
\]
and
\[
 N_{r}=\{x\in \GL_{r}\mid (x-1)X_i\subset
  X_{i-1},\,i=1,2,\cdots, r\},
\]
which are maximal unipotent subgroups of $\GL_{r+1}$ and $\GL_r$,
respectively.

\vsp We now describe other salient features of the Bessel group
$S_r$. It is a subgroup of $P_r$:
\begin{equation}
\label{SrPr}
   S_r=N_{P_r}\rtimes (G_0\times N_r)\subset P_r=N_{P_r}\rtimes (G_0'\times \GL_r).
\end{equation}
Although $P_r$ is not a subgroup of $P_{r+1}'$, we have that
$S_r\subset P_{r+1}'$ and the quotient map $P_{r+1}'\rightarrow G_0
\times \GL_{r+1}$ induces a surjective homomorphism
\begin{equation}
\label{deta'}
 \tilde{\eta}_r : S_r \twoheadrightarrow G_0\times N_{r+1}.
\end{equation}
It is elementary to check that every character on $S_r$ descends to
one on $G_0\times N_{r+1}$, and it is generic if and only if its
descent to $G_0\times N_{r+1}$ has generic restriction to $N_{r+1}$,
in the usual sense.

Let $\chi_{S_r}$ be a generic character of $S_r$, as in Theorem
\ref{main}. Write
\begin{equation}\label{gc}
 \chi_{S_r}=(\chi_{G_0}\otimes \psi_{r+1})\circ \tilde{\eta}_r,
\end{equation}
where $\chi_{G_0}$ is a character on $G_0$, and $\psi_{r+1}$ is a
generic character on $N_{r+1}$. Throughout this article, we always
assume that $\psi_{r+1}$ is unitary. Otherwise the Hom space in
Theorem \ref{main} is trivial, due to the moderate growth condition
on the representation $\pi$.

\subsection{Induced representations of $G'$}
\label{subind} Let $\pi_0$ and $\sigma$ be irreducible
representations of $G_0$ and $\GL_{r+1}$ in the class $\CFH$,
respectively. Write
\[
  \rho:=\pi_0\widehat{\otimes}\sigma,
\]
which is a representation of
$G_0\times \GL_{r+1}$ in the class $\CFH$.

Put
\[
   d_\K:=\left\{
            \begin{array}{l}
              1, \quad \textrm{if $\K$ is a field},\\
              2, \quad \textrm{otherwise},\
            \end{array}
        \right.
\]
and
\[
  \K^\times_+=\left\{
            \begin{array}{l}
              \R_+^\times, \quad \quad \quad \ \ \textrm{if $d_\K=1$},\smallskip\\
               \R_+^\times\times  \R_+^\times, \quad \textrm{otherwise}.\
            \end{array}
        \right.
\]
Denote by
\[
  \abs{\cdot}:\K^\times\twoheadrightarrow \K^\times_+
\]
the map of taking componentwise absolute values. For all $a\in
\K^\times_+$ and $s\in \C^{d_\K}$, put
\[
  a^s:=a_1^{s_1}\, a_2^{s_2}\in \C^\times, \quad \textrm{if }\,d_\K=2,\,\, a=(a_1,a_2),\, s=(s_1,s_2).
\]
If $d_\K=1$, $a^s\in \C^\times$ retains the usual meaning.

\vsp

We now define certain representations of $G'$ in the class $\CFH$ which are induced from the parabolic subgroup $P_{r+1}'$.
For every $s\in \C^{d_\K}$, denote by $\pi_s'$ the space of all
smooth functions $f: G'\rightarrow \rho $ such that
\[
   f(n'g m x)=\chi_{G_0}(g)^{-1}\,\abs{\det(m)}^s\,\rho(g m)(f(x)),
\]
for all $n'\in N_{P_{r+1}'}$,  $g\in G_0$,  $m\in \GL_{r+1}$, $x\in
G'$. (We introduce the factor $\chi_{G_0}(g)^{-1}$ for convenience
only.)

By using Langlands classification and the result of Speh-Vogan
\cite[Theorem 1.1]{SV}, we have
\begin{prp}
\label{irr} The representation $\pi'_s$ is irreducible except for a
measure zero set of $s\in \C^{d_\K}$.
\end{prp}

\subsection{The integral $I_{\mu}$} \label{defint}

Recall that $\psi_{r+1}$ is the generic unitary character of
$N_{r+1}$ as in (\ref{gc}). Assume that the representation $\sigma$
of $\GL_{r+1}$ is $\psi_{r+1}^{-1}$-generic, namely there exists a
nonzero continuous linear functional
\[
  \lambda: \sigma\rightarrow \C
\]
such that
\[
  \lambda(\sigma(m)u)= \psi_{r+1}(m)^{-1} \,\lambda(u), \quad m\in
  N_{r+1}, \,u\in \sigma.
\]
We fix one such $\lambda$. Define a continuous
linear map $\Lambda$ by the formula
\[
  \begin{array}{rcl}
  \Lambda: \rho=\pi_0\widehat{\otimes} \sigma&\rightarrow & \pi_0\\
   u\otimes v&\mapsto& \lambda(v) u.
  \end{array}
\]

Let $\pi$ be an irreducible representation of
$G$ in the class $\CFH$, as in Theorem \ref{main}, and let
\[
  \mu: \pi\times \pi_0\rightarrow \C
\]
be a Bessel functional, namely a continuous bilinear map which
corresponds to an element of
\[
  \Hom_{S_r}(\pi\widehat{\otimes} \pi_0,
  \chi_{S_r}).
\]

\begin{lemp}\label{srinv}
For every $s\in\C^{d_\K}$, $u\in \pi$ and $f\in \pi'_s$, the smooth
function
\[
  g\mapsto  \mu(\pi(g)u,\,\Lambda(f(g)))
\]
on $G$ is left invariant under $S_r$.
\end{lemp}
\begin{proof}
Let $x\in G$ and $b\in S_r\subset P_{r+1}'$. Write
\[
  b=n'gm,\quad n'\in N_{P_{r+1}'}, \,g\in G_0, \,m\in N_{r+1}.
\]
Then
\begin{eqnarray*}
   \Lambda(f(bx)))&=&\chi_{G_0}(g)^{-1}\Lambda(\rho(gm)(f(x)))\\
   &=&\chi_{G_0}(g)^{-1}\psi_{r+1}(m)^{-1}\pi_0(g)(\Lambda(f(x)))\\
   &=&\chi_{S_r}(b)^{-1}\,\pi_0(g)(\Lambda(f(x))),\\
\end{eqnarray*}
and therefore,
\begin{eqnarray*}
  \mu(\pi(bx)u,\,\Lambda(f(bx)))
  &=&\chi_{S_r}(b)^{-1}\mu(\pi(b)\pi(x)u,\,\pi_0(g)(\Lambda(f(x))))\\
     &=&\mu(\pi(x)u,\,\Lambda(f(x))).\\
   \end{eqnarray*}
The last equality holds as $b$ maps to $g$ under the quotient map
$S_r\twoheadrightarrow G_0$, and $\pi _0$ is viewed as a
representation of $S_r$ via inflation.
\end{proof}

Write
\begin{equation}\label{integral}
  I_\mu(f,u):=\int_{S_r\backslash G} \mu(\pi(g)u,\,\Lambda(f(g)))
  \,dg,\quad f\in \pi'_s, \,u\in\pi,
\end{equation}
where $dg$ is a right $G$-invariant positive measure on
$S_r\backslash G$. It is clear that
\[
  I_\mu(\pi_s'(g)f, \pi(g)u)=I_\mu(f, u)
\]
for all $g\in G$ whenever the integrals converge absolutely.

\subsection{Proof of Theorem \ref{main}}\label{pro}

We shall postpone the proof of the following proposition to Section
\ref{pnonzero}.

\begin{prpp}\label{nonzero}
If $\mu\neq 0$, then there is an element $f_\rho\in \pi'_s$ and a
vector $u_\pi \in \pi$ such that the integral
$I_\mu(f_\rho,u_\pi)$ converges absolutely, and yields a nonzero
number.
\end{prpp}

Denote by
\[
  \Re: \C^{d_\K}\rightarrow \R^{d_\K}
\]
the map of taking real parts componentwise. If $a\in \R^{d_\K}$ and
$c\in \R$, by writing $a>c$, we mean that each component of $a$ is
$>c$.

The proof of the following proposition will be given in Section
\ref{conv2} after preparation in Section \ref{int}.

\begin{prpp}\label{converge}
There is a real number $c_\mu$ such that for all $s\in \C^{d_\K}$
with $\Re(s)>c_\mu$, the integral $I_\mu(f,u)$ converges absolutely
for all $f\in \pi'_s$ and all $u\in \pi$, and $I_\mu$ defines a
continuous linear functional on $\pi'_s\widehat \otimes \pi$.
\end{prpp}

\vsp We now complete the proof of Theorem \ref{main}. We are given
$\pi, \pi _0$ and a generic character $\chi_{S_r}$ of $S_r$ as in
Equation \eqref{gc}. As noted there, we may assume that the generic
character $\psi_{r+1}$ of $N_{r+1}$ is unitary. As is well known,
there exists an irreducible representation
$\sigma $ of $\GL_{r+1}$ in the class $\CFH$ which is $\psi_{r+1}^{-1}$-generic (it follows from
\cite[Theorem 9.1]{CHM}, for example). For each $\mu \in
\Hom_{S_r}(\pi\widehat{\otimes} \pi_0,\chi_{S_r})$, we may therefore
define the integral $I_{\mu}$, as in Section \ref{defint}.

Let $F_r$ be a finite dimensional subspace of
$\Hom_{S_r}(\pi\widehat{\otimes} \pi_0,\chi_{S_r})$.  By Proposition
\ref{converge}, there exists a real number $c_{F_r}$ such that for
all $\mu \in F_r$ and all $s\in \C^{d_\K}$ with $\Re(s)>c_{F_r}$,
the integral $I_\mu(f,u)$ converges absolutely for all $f\in \pi'_s$
and all $u\in \pi$, and defines a continuous linear functional on
$\pi'_s\widehat \otimes \pi$.

By Proposition \ref{irr}, we may choose one $s$ with
$\Re(s)>c_{F_r}$ and $\pi'_s$ irreducible. By Proposition
\ref{nonzero}, we have a linear embedding
\[
  F_r \hookrightarrow \Hom_G(\pi_s'\widehat{\otimes}
  \pi,\C), \quad \mu\mapsto I_\mu.
\]
The later space is at most one dimensional by \cite[Theorem A]{SZ},
and so is $F_r$. This proves Theorem \ref{main}. \qed

\section{Proof of Proposition \ref{nonzero}}
\label{pnonzero}

We continue with the notation of the last section and assume that
$\mu\neq 0$. By absorbing the concerned characters into the
representations $\pi_0$ and $\sigma$, we may and we will assume in
this section that $\chi_{G_0}=1$ and $s=0\in \C^{d_\K}$.

\vsp Let $\bar{N}_{P_r}$ be the unipotent subgroup of $G$ which is
normalized by $G_0'\times \GL_r $ so that $(G_0'\times \GL_r )
\bar{N}_{P_r}$ is a parabolic subgroup opposite to $P_r$. Then
\[P_r\bar{N}_{P_r} \text{ is open in $G$}.\]

Recall that the Bessel group $S_r$ is a subgroup of $P_r$:
\[
   S_r=N_{P_r}\rtimes (G_0\times N_r)\subset P_r=N_{P_r}\rtimes (G_0'\times \GL_r).
\]

We shall need to integrate over $S_r\backslash G$, and thus over the
following product space
\[N_r\backslash \GL_r\times (G_0\backslash G_0')\times
\bar{N}_{P_r}.\]

\subsection{A nonvanishing lemma on $G_0\backslash G_0'$}

\begin{lem}\label{upi}
There is a vector $u_\pi\in \pi$ and a smooth function $f_{\pi_0}:
G_0'\rightarrow \pi_0$, compactly supported modulo $G_0$ such that
\[
  f_{\pi_0}(gg')=\pi_0(g)f_{\pi_0}(g'), \quad g\in G_0, \,g'\in G_0',
\]
and
\begin{equation}\label{nonzero1}
   \int_{G_0\backslash G_0'} \mu(\pi(g') u_\pi, f_{\pi_0}(g'))\ dg'\neq 0.
\end{equation}
\end{lem}

\begin{proof}
Pick $u_\pi\in \pi$ and $u_{\pi_0}\in \pi_0$ so that
\[
  \mu(u_{\pi}, u_{\pi_0})=1.
\]
Let $A'$ be a submanifold of $G_0'$ such that the multiplication map
$G_0\times A'\rightarrow G_0'$ is an open embedding, and
\[
  \Re (\mu(\pi(a) u_{\pi}, u_{\pi_0}))>0, \quad a\in A'.
\]
Let $\phi_0$ be a compactly supported nonnegative and nonzero smooth
function on $A'$. Put
\[
  f_{\pi_0}(g'):=\left\{
                     \begin{array}{ll}
                       \phi_0(a) \pi_0(g) u_{\pi_0},\quad& \textrm{if }g'=g a\in G_0 A',\\
                       0, \quad&\textrm{otherwise},
                      \end{array}
                  \right.
\]
which clearly fulfills all the desired requirements.

\end{proof}

\subsection{Whittaker functions on $\GL_r$}

Fix $u_\pi$ and $f_{\pi_0}$ as in Lemma \ref{upi}. Set
\[
  \Phi(m, \bar n):=\int_{G_0\backslash G_0'} \mu(\pi(mg'\bar{n}) u_{\pi}, f_{\pi_0}(g'))\
  dg',
\]
which is a smooth function on $\GL_r\times \bar{N}_{P_r}$. It is
nonzero since $\Phi(1,1)\neq 0$ by (\ref{nonzero1}). Note that
$\mu$ is $\chi_{S_r}$-equivariant, the representation $\pi_0$ of
$S_r$ has trivial restriction to $N_r$, and $\chi_{S_r}$ and
$\psi_{r+1}$ have the same restriction to $N_r$. Therefore, we
have
\[
   \Phi(lm,\bar n)=\psi_{r+1}(l)\Phi(m, \bar n), \quad l\in N_{r}, \, m\in
   \GL_r,\, \bar n\in\bar{N}_{P_r}.
\]

Let $W_r$ be a smooth function on $\GL_r$ with compact support
modulo $N_r$ such that
\begin{equation}\label{whe}
  W_r(lm)=\psi_{r+1}(l)^{-1} W_r(m), \quad l\in N_r, \, m\in \GL_r.
\end{equation}

The following lemma is due to Jacquet and Shalika (\cite[Section
3]{JS}, see also \cite[Section 4]{Cog}).

\begin{leml}\label{usigma}
For every $W_r$ as above, there is a vector $u_\sigma\in \sigma$
such that
\[
  W_r(m)=\lambda(\sigma(m)u_\sigma), \quad m\in \GL_r.
\]
\end{leml}
Let $\phi_{\bar N}$ be a smooth function on $\bar{N}_{P_r}$ with
compact support. Pick $W_r$ and $\phi_{\bar N}$ appropriately so
that
\begin{equation}\label{intnonzero}
   \int_{(N_r\backslash \GL_r)\times \bar{N}_{P_r}} \delta_{P_r}^{-1}(m) \Phi(m,\bar
     n) W_r(m) \phi_{\bar N}(\bar n)\, dm\,d\bar{n}\neq 0.
\end{equation}
Here and as usual, we denote by
\[
  \delta_H: h\mapsto \abs{\det(\Ad_h)}
\]
the modular character of a Lie group $H$.

\subsection{The construction of $f_\rho$}

Note that
\begin{equation}\label{intgroup}
   P'_{r+1}\cap G=N_{P_r}\rtimes (G_0\times \GL_r).
\end{equation}
By counting the dimensions of the concerned Lie groups, we check
that the multiplication map
\begin{equation}\label{submersion}
  P'_{r+1}\times G\rightarrow G' \,\textrm{ is a submersion.}
\end{equation}
From \eqref{intgroup} and \eqref{submersion}, we see that the
multiplication map
\[
  \iota_{G'}: (N_{P_{r+1}'}\rtimes \GL_{r+1})\times (G_0'\ltimes
  \bar{N}_{P_r})\rightarrow G'
\]
is an open embedding.

Put
\[
  f_{\rho}(x):=\left\{
                     \begin{array}{ll}
                       \phi_{\bar N}(\bar n)\, f_{\pi_0}(g')
  \otimes (\sigma(m) u_\sigma),\quad& \textrm{if }x=\iota_{G'}(n',m, g', \bar{n}),\\
                       0, \quad&\textrm{if $x$ is not in the image of $\iota_{G'}$},
                      \end{array}
                  \right.
\]
where $u_\sigma$ is as in Lemma \ref{usigma}. Then $f_\rho\in
\pi_s'$ (recall that $s$ is assumed to be $0$).

Finally, we have that
\begin{eqnarray*}
 &&\quad I_\mu(f_\rho, u_\pi)\\
 &&=\int_{S_r\backslash G}\mu(\pi(x)u_\pi,\,\Lambda(f_\rho(x)))\,dx\\
 &&=\int_{(N_r\backslash \GL_r)\times (G_0\backslash G_0')\times \bar{N}_{P_r}}
 \mu(\pi(m g' \bar{n})u_\pi,\,\Lambda(f_\rho(m g'\bar{n})))
 \delta_{P_r}^{-1}(m)\, dm\,dg'\,d\bar{n}\\
 &&=\int_{(N_r\backslash \GL_r)\times (G_0\backslash G_0')\times \bar{N}_{P_r}}
 \delta_{P_r}^{-1}(m) \lambda(\sigma(m) u_\sigma) \phi_{\bar N}(\bar n)\\
 &&\phantom{\int_{(N_r\backslash \GL_r)\times (G_0\backslash G_0')\times
 \overline{N}_{P_r}}} \cdot \mu(\pi(m g'\bar{n})u_\pi,\,f_{\pi_0}(g'))
 \, dm\,dg'\,d\bar{n}\\
 &&=\int_{(N_r\backslash \GL_r)\times \bar{N}_{P_r}}
 \delta_{P_r}^{-1}(m) \Phi(m, \bar n) W_r(m) \phi_{\bar N}(\bar n) dm\,d\bar{n},\\
\end{eqnarray*}
which converges to a nonzero number by \eqref{intnonzero}. This
finishes the proof of Proposition \ref{nonzero}.

\section{Another integral formula on $S_r\backslash G$}\label{int}

This section and the next section are devoted to a proof of
Proposition \ref{converge} in the case when $E_0':=E_0\oplus \K
v_0'$ is isotropic, i.e., when $E'_0$ has a torsion free isotropic
vector. The anisotropic case is simpler and is left to the reader.
We first develop some generalities in the following two subsections.

\subsection{Commuting positive forms}

Let $F$ be a free $\K$-module of finite rank. By a positive form on
$F$, we mean a $\R$-bilinear map
\[
   [ \,,\,]_F :F\times F\rightarrow \K
\]
satisfying
\[
  [ u,v]_F=\overline{[v,u]_F}, \quad [ au,v]_F=a[u,
     v]_F,\quad a\in \K,\, u,v\in F.
\]
and
\[
[u,u]_F\in \K^\times_+\quad \textrm{for all torsion free }u\in F.
\]
Here $\bar a\in \K$ denotes the componentwise complex conjugation of
$a$, for $a\in \K$.

Now further assume that $F$ is a hermitian $\K$-module, i.e., a
non-degenerate hermitian form $\la\,,\,\ra_F$ (with respect to
$\tau$) on $F$ is given. We say that the positive form $[\,,\,]_F$
is commuting (with respect to $\la\,,\,\ra_F$) if
\[
   \theta_F^2=1,
\]
where $\theta_F:F\rightarrow F$ is the $\R$-linear map specified by
\begin{equation}\label{defs}
  [ u,v]_F=\la u, \theta_F v\ra_F, \quad u,v\in F.
\end{equation}

The following lemma is elementary.

\begin{lem}
Up to the action of $\oU(F)$, there exists a unique commuting
positive form  on $F$.
\end{lem}
\begin{proof}
We check the case of complex orthogonal groups, and leave other
cases to the reader. So assume that $(\K,\tau)=(\C, 1_\C)$. Then
\[
  [\,,\,]_F\mapsto \textrm{the eigenspace of $\theta_F$ of eigenvalue
  $1$}\]
defines a $\oU(F)$-equivariant bijection:
\[
  \begin{array}{l}
        \,\quad\{\,\textrm{commuting positive
                 forms on $F$}\,\}\medskip\\
             \leftrightarrow
                \{\,\textrm{real forms $F_0$ of $F$ so that $\la\,,\,\ra_F|_{F_0\times F_0}$ is real valued and positive
                definite}\,\}.
   \end{array}
\]
The assertion follows immediately. \end{proof}

\subsection{A Jacobian}
Now fix a commuting positive form $[\,,\,]_F$, and denote by
$\oK(F)$ its stabilizer in $\oU(F)$ (which is also the centralizer
of $\theta _F$ in $\oU(F)$). Then $\oK(F)$ is a maximal compact
subgroup of $\oU(F)$. Write
\[
  F=F_+\oplus F_-,
\]
where $F_+$ and $F_-$ are eigenspaces of $\theta_F$ of eigenvalues
$1$ and $-1$, respectively.

With the preparation of the commuting positive forms, we set
\[
  S_F:=\{u+v\mid u\in F_+,\,v\in F_-,\, [u,u]_F=[v,v]_F=1,
  \,[u,v]_F= 0\}.
\]
Assume that $F$ is isotropic, i.e., there is a torsion free vector
of $F$ which is isotropic with respect to $\la\,,\,\ra_F$. This is
the case of concern. Then $S_F$ is nonempty. It is easy to check
that $\oK(F)$ acts transitively on $S_F$. According to \cite{Sh},
$S_F$ is in fact a Nash-manifold. Furthermore, it is a Riemannian
manifold with the restriction of the metric
\[
  \frac{1}{\dim_\R \K} \, \tr_{\K/\R}[\, ,\,]_F.
\]

Write
\[
  \Gamma_{F, -1}:=\{u\in F\mid \la u,u\ra_F=-1\},
\]
which is a Nash-manifold. It is also a pseudo-Riemannian manifold
with the restriction of the metric
\[
  \frac{1}{\dim_\R \K} \, \tr_{\K/\R}\la\, ,\,\ra_F.
\]

Equip $\R^\times_+$ with the invariant Riemannian metric so that the
tangent vector $t\frac{d}{dt}$ at $t\in \R^\times_+$ has length $1$.
As a product of one or two copies of $\R^\times_+$, $\K^\times_+$ is
again a Riemannian manifold.

Define a map
\[
  \begin{array}{rcl}
  \eta_F:  S_F \times\K^\times_+&\rightarrow& \Gamma_{F,-1},\smallskip\\
    (w, t)&\mapsto & \frac{t-t^{-\tau}}{2} u+\frac{t+t^{-\tau}}{2}v,
  \end{array}
\]
where
\[
  w=u+v, \quad u\in F_+, v\in F_-.
\]
Note that the domain and the range of the smooth map $\eta_F$ have
the same real dimension. Denote by $J_{\eta_F}$ the Jacobian of
$\eta_F$ (with respect to the metrics defined above), which is a
nonnegative continuous function on $S_F \times\K^\times_+$. Since
$\eta_F$ and all the involved metrics are semialgebraic,
$J_{\eta_F}$ is also semialgebraic (see \cite{Sh} for the notion
of semialgebraic maps and Nash maps). Note that $\oK(F)$ acts
transitively on $S_F$ (and trivially on $\K^\times_+$), $\eta_F$
and all the involved metrics are $\oK(F)$-equivariant. Therefore,
there is a nonnegative continuous semialgebraic function $J_F$ on
$\K^\times_+$ such that
\[
  J_{\eta_F}(w,t)=J_F(t), \quad w\in S_F,\, t\in \K^\times_+.
\]

Denote $\con(X)$ the space of continuous functions on any
(topological) space $X$.

\begin{leml}\label{intgamma}
For $\phi \in \con (\Gamma_{F,-1})$, one has that
\[
  \int_{\Gamma_{F,-1}} \phi(x)\,dx=\frac{1}{2}\int_{\K^\times_+}J_F(t)
  \int_{S_F} \phi(\eta_{F}(w,t))\,dw\, d^\times t,
\]
where $dx$, $dw$ and $d^\times t$ are the volume forms associated to
the respective metrics.
\end{leml}

\begin{proof}
For every $t\in \K^\times_+$, write
\begin{equation}
\label{lr}
  \la t\ra:=\left\{
               \begin{array}{ll}
                t, \quad & \textrm{if $d_\K=1$,}\\
                t_1 t_2, \quad & \textrm{if $d_\K=2$ and $t=(t_1,t_2)$.}\\
               \end{array}
        \right.
\end{equation}

Write
\[
  \Gamma_{F,-1}(1)=\{u\in \Gamma_{F,-1}\mid \la [u,u]_F\ra=1\},
\]
which is a closed submanifold of $\Gamma_{F,-1}$ of measure zero.
One checks case by case that $\eta_F$ induces diffeomorphisms from
both
\[
 S_F\times \{t\in \K^\times_+\mid \la t\ra >1\}\quad \textrm{and}\quad S_F\times
   \{t\in \K^\times_+\mid \la t\ra <1\}
\]
onto $\Gamma_{F,-1}\setminus \Gamma_{F,-1}(1)$. The lemma then
follows.
\end{proof}

\subsection{A preliminary integral formula on $G_0\backslash G_0'$}

We recall the notations of Section \ref{pre}. The isotropic
condition on $E_0'=E_0\oplus \K v_0'$ (which we assume as in the
beginning of this section) ensures that there is a vector $v_0\in
E_0$ such that
\[
  \la v_0,v_0\ra=1.
\]
Denote by $Z_0$ its orthogonal complement in $E_0$. Then $E$ is an
orthogonal sum of four submodules:
\begin{equation}\label{decome}
  E=(X_r\oplus Y_r)\oplus Z_0\oplus \K v_0\oplus \K v_0'.
\end{equation}
Fix a commuting positive form $[\,,\,]_{E}$ on $E$ so that
(\ref{decome}) is an orthogonal sum of five submodules with respect
to $[\,,\,]_{E}$. Recall that
\[
  G:=\oU(E),\quad G_0':=\oU(E_0'),\quad G_0=\oU(E_0).
\]
Put
\[
    K:=\oK(E),\quad K_0':=\oK(E_0'),\quad K_0:=\oK(E_0).
\]

For every $t\in\K^\times_+$, denote by $g_t\in G_0'$ the element
which is specified by
\begin{equation}\label{defgt}
 \left\{
   \begin{array}{l}
    g_t(v_0+v_0')=t(v_0+v_0'),\\
    g_t(v_0-v_0')=t^{-\tau}(v_0-v_0'), \,\,\textrm{and}\\
    g_t|_{X_r\oplus Y_r\oplus Z_0}=\textrm{the identity map}.
   \end{array}
   \right.
\end{equation}

We use the results of the last two subsections to prove the
following lemma.
\begin{leml}\label{lemint}
For $\phi \in \con (G_0\backslash G_0')$, we have
\[
   \int_{G_0\backslash G_0'} \phi(x) \,dx=\int_{\K_+^\times} J_{E_0'}(t) \int_{K_0'}
   \,\phi(g_t k) \,dk\,d^\times t,
\]
where $dk$ is the normalized haar measure on $K_0'$, and $dx$ is a
suitably normalized $G_0'$-invariant positive measure on
$G_0\backslash G_0'$.
\end{leml}

\begin{proof}
By first integrating over $K_0'$, we just need to show that
\begin{equation}\label{intq}
   \int_{G_0\backslash G_0'} \phi(x) \,dx=\int_{\K_+^\times} J_{E_0'}(t) \, \phi(g_t) \,d^\times t,
\quad \phi\in \con(G_0\backslash G_0'/K_0').
\end{equation}

We identify $G_0\backslash G_0'$ with $\Gamma_{E_0',-1}$ by the
map $g\mapsto g^{-1}v_0'$. Note that $v_0+v_0'\in S_{E_0'}$ and
$G_0 g_t$ is identified with $\eta_{E_0'}(v_0+v_0',t^{-1})$. The
measure $dx$ is identified with a constant $C$ multiple of the
metric measure $dy$ on $\Gamma_{E_0',-1}$.

Let
\[
  \phi\in \con(G_0\backslash G_0'/K_0')=\con (K_0'\backslash\Gamma_{E_0',-1}).
\]
Then the function $\phi(\eta_{E_0'}(w,t^{-1}))$ is independent of
$w\in S_{E_0'}$. Also note that
\[
  J_{E_0'}(t)=J_{E_0'}(t^{-1}),\qquad t\in \K^\times_+.
\]
Therefore by Lemma \ref{intgamma}, we have
\begin{eqnarray*}
  \int_{G_0\backslash G_0'} \phi(x) \,dx
  && =C \int_{\Gamma_{E_0',-1}} \phi(y)\,dy\\
  && =\frac{1}{2}C \int_{\K^\times_+}J_{E_0'}(t)
  \int_{S_{E_0'}} \phi(\eta_{E_0'}(w,t))\,dw\, d^\times t\\
  && =\frac{1}{2}C \int_{\K^\times_+}J_{E_0'}(t^{-1})
  \int_{S_{E_0'}} \phi(\eta_{E_0'}(w,t^{-1}))\,dw\, d^\times t\\
  && =\frac{1}{2}C \int_{\K^\times_+}J_{E_0'}(t)
  \int_{S_{E_0'}} \phi(\eta_{E_0'}(v_0+v_0',t^{-1}))\,dw\, d^\times t\\
  && =\frac{1}{2}C \int_{\K^\times_+}J_{E_0'}(t) \phi(g_t)
  \int_{S_{E_0'}}1 \,dw\, d^\times t.\\
  \end{eqnarray*}
We finish the proof by putting
\[
  C:=2\left(\int_{S_{E_0'}}1 \,dw\right)^{-1}.
\]

\end{proof}

\subsection{The integral formula on $S_r\backslash G$}\label{defat}
Denote by $B_r$ the Borel subgroup of $\GL_r$ stabilizing the flag
\[
   0=X_0\subset X_1\subset\cdots \subset X_r.
\]
For every $\bt=(t_1,t_2,\cdots, t_r)\in (\K_+^\times)^r$, denote by
$a_\bt$ the element of $\GL_r$ whose restriction to
\[
  \{v\in X_i\mid [v, X_{i-1}]_{E}=0\}
\]
is the scalar multiplication by $t_i$, for $i=1,2,\cdots, r$.

\begin{prpl}\label{intf}
For $\phi \in \con (S_r\backslash G)$, one has that
\[
   \int_{S_r\backslash G} \phi(g)
   \,dg=\int_{(\K^\times_+)^r \times \K^\times_+\times K}
  \phi(a_\bt  g_t k ) \delta_{P_r}^{-1}(a_{\bt})\, \delta_{B_{r}}^{-1}(a_{\bt})\, J_{E_0'}(t)\,d^\times \bt \,d^\times t\,dk, \\
\]
where $dg$ is a suitably normalized right $G$-invariant measure on
$S_r\backslash G$.
\end{prpl}

\begin{proof}
Write $K_r=K\cap \GL_r$. Then we have
\begin{eqnarray*}
  \int_{S_r\backslash G} \phi(g) \,dg
  &&=\int_{(N_{P_r} N_r G_0)\backslash (N_{P_r} \GL_r G_0' K)}
   \phi(g) \,dg\\
  &&=\int_{(N_r\backslash \GL_r)\times (G_0\backslash G_0')\times K}
  \phi(m g' k) \delta_{P_r}^{-1}(m)\,dm\,dg'\, dk\\
  &&=\int_{(\K^\times_+)^r \times  K_r\times (G_0\backslash G_0')\times K}
  \phi(a_\bt l g' k) \delta_{P_r}^{-1}(a_{\bt})\, \delta_{B_r}^{-1}(a_{\bt})\,d^\times \bt \,dl \,dg'\, dk\\
  &&=\int_{(\K^\times_+)^r \times (G_0\backslash G_0')\times K}
  \phi(a_\bt  g' k) \delta_{P_r}^{-1}(a_{\bt})\, \delta_{B_r}^{-1}(a_{\bt})\,d^\times \bt \,dg'\, dk\\
  &&\qquad (\textrm{$l\in K_r\subset \GL_r$ commutes with $g'\in G_0'$})\\
   &&=\int_{(\K^\times_+)^r \times \K^\times_+\times K_0'\times K}
  \phi(a_\bt  g_t l k) \delta_{P_r}^{-1}(a_{\bt})\, \delta_{B_r}^{-1}(a_{\bt})\, J_{E_0'}(t)\,d^\times \bt \,d^\times t\,dl\, dk \\
  &&\qquad (\textrm{By Lemma \ref{lemint}})\\
  &&=\int_{(\K^\times_+)^r \times \K^\times_+\times K}
  \phi(a_\bt  g_t k ) \delta_{P_r}^{-1}(a_{\bt})\, \delta_{B_r}^{-1}(a_{\bt})\, J_{E_0'}(t)\,d^\times \bt \,d^\times t\,dk. \\
\end{eqnarray*}

\end{proof}

\section{Proof of Proposition \ref{converge}}\label{conv2}

\subsection{An Iwasawa decomposition}
Recall that we have a hermitian $\K$-module
\begin{equation}\label{decep}
  E'=X_r\oplus Y_r\oplus Z_0\oplus \K v_0\oplus \K v_0'\oplus \K v'.
\end{equation}
Equip it with the commuting positive form $[\,,\,]_{E'}$ which
extends $[\,,\,]_E$ and makes $v'$ and $E$ perpendicular. Also
recall that $G':=\oU(E')$. Put $K':=\oK(E')$.

Write
\begin{equation}\label{dece3}
  E_3=\K v_0\oplus \K v_0'\oplus \K v',
\end{equation}
and denote by $N_{E_3}$ the unipotent radical of the Borel subgroup
of $\oU(E_3)$ stabilizing the line $\K(v_0'+v')$. For every $t\in
\K^{\times}_+$, denote by $b_t\in \oU(E_3)$ the element specified by
\begin{equation}\label{defbt}
\left\{
   \begin{array}{l}
     b_t(v_0)=v_0,\\
     b_t(v'_0+v')=t (v'_0+v'),\\
     b_t(v'_0-v')=t^{-\tau} (v'_0-v').\\
   \end{array}
\right.
\end{equation}
For every $t\in \{t\in \K^{\times}_+\mid t t^\tau=1\}$, denote by
$c_t\in \oU(E_3)$ the element specified by
\[
\left\{
   \begin{array}{l}
     c_t(v_0)=t v_0,\\
     c_t(v'_0)=v'_0,\\
     c_t(v')=v'.\\
   \end{array}
\right.
\]
Recall the element $g_t\in G'_0 \subset G'$ in (\ref{defgt}). Note
that it also stays in $\oU(E_3)$. By Iwasawa decomposition, we write
\begin{equation}\label{decgt}
  g_t=c_{t''} n_t b_{t'}  k_t, \quad n_t\in N_{E_3},\, k_t\in
  \oK(E_3).
\end{equation}
Then both $t'$ and $t''$ are Nash functions of $t$.

\begin{lem}
 One has that
\[
  t'=2(t^{-2} + t^{2\tau}+2)^{-\frac{1}{2}}.
\]
\end{lem}
\begin{proof}
Note that $v_0, v_0',v'$ is an orthonormal basis of $E_3$ with
respect to $[\,,\,]_{E'}$. We have that
\begin{eqnarray*}
  &&\quad [g_t^{-1}(v_0'+v'),g_t^{-1}(v_0'+v')]_{E'}\\
  && =[k_t^{-1}  b_{t'}^{-1} n_t^{-1}c_{t''}^{-1}(v_0'+v'),k_t^{-1}  b_{t'}^{-1} n_t^{-1}c_{t''}^{-1}(v_0'+v')]_{E'}\\
  &&=[t'^{-1}(v_0'+v'),t'^{-1}(v_0'+v')]_{E'}\\
  &&=2 t'^{-2}.
\end{eqnarray*}

On the other hand,
\[
  g_t^{-1}(v_0'+v')=\frac{t^{-1}-t^\tau}{2}v_0+\frac{t^{-1}+t^\tau}{2}v'_0+v',
\]
and
\begin{eqnarray*}
  &&\quad [g_t^{-1}(v_0'+v'),g_t^{-1}(v_0'+v')]_{E'}\\
  &&=\left(\frac{t^{-1}-t^\tau}{2}\right)^2+\left(\frac{t^{-1}+t^\tau}{2}\right)^2+1\\
  &&=\frac{t^{-2}+t^{2\tau}+2}{2}.
\end{eqnarray*}
Therefore the lemma follows.
\end{proof}

\subsection{Majorization of Whittaker functions}

We define a norm function on $G'$ by
\[
  ||g||:=\max\{(\la[gu,gu]_{E'}\ra)^{\frac{1}{2}}\mid u\in
  E',\,\la [u,u]_{E'}\ra=1\}, \quad g\in G',
\]
where $\la \cdot \ra$ is as in (\ref{lr}).

 For every $\tilde{\bt}=(t_1,t_2,\cdots,t_r,t_{r+1})\in
(\K^\times_+)^{r+1}$, write
\begin{equation}\label{xi}
  \xi(\tilde{\bt})=\left\{
                      \begin{array}{ll}
                       \prod_{i=1}^{r} (1+\frac{t_i}{t_{i+1}}), \quad & \textrm{if }  d_\K=1,\smallskip\\
                       \prod_{i=1}^{r} (1+\frac{t_{i,1}}{t_{i+1,1}})\times \prod_{i=1}^{r} (1+\frac{t_{i,2}}{t_{i+1,2}}), \quad & \textrm{if }
                       d_\K=2 \textrm{ and } t_i=(t_{i,1},t_{i,2}).
                      \end{array}
                    \right.
\end{equation}
Write
\[
  a_{\tilde{\bt}}=a_\bt b_{t_{r+1}}\in \GL_{r+1}, \quad
  \textrm{with } \bt=(t_1,t_2,\cdots,t_r).
\]
Recall that $a_\bt$ is defined in Section \ref{defat} and $b_t$ is
defined in \eqref{defbt}.

Following \cite[Proposition 3.1]{Jac}, we have
\begin{leml}\label{whittaker}
Let notations be as in Section \ref{subind}. Let $c_\rho$ be a
positive number, $\abs{\cdot}_{\pi _0}$ a continuous seminorm on
$\pi_0$, and $\abs{\cdot}_{\rho,0}$ a continuous seminorm on $\rho$.
Assume that
\[
    \abs{\Lambda(\rho(g)u)}_{\pi _0}\leq ||g||^{c_\rho} \abs{u}_{\rho,0}, \quad g\in G_0\times \GL_{r+1}, \, u\in \rho.
\]
Then for any positive integer $N$, there is a continuous seminorm
$\abs{\cdot}_{\rho,N}$ on $\rho$ such that
\[
  \abs{\Lambda (\rho(a_{\tilde{\bt}})u)}_{\pi _0}\leq \xi(\tilde{\bt})^{-N} ||a_{\tilde{\bt}}||^{c_\rho}\, \abs{u}_{\rho,N},
  \quad
  \tilde{\bt}\in (\K^\times_+)^{r+1}, \, u\in \rho.
\]
\end{leml}

\begin{proof}To ease the notation, we
assume that $d_\K=1$. The other case is proved in the same way. For
every $i=1,2,\cdots, r$, let $Y_i$ be a vector in the Lie algebra of
$\GL_{r+1}$ so that
\[
  \Ad_{a_{\tilde{\bt}}}Y_i=\frac{t_i}{t_{i+1}} Y_i, \qquad
  \tilde{\bt}=(t_1,t_2,\cdots,t_{r+1}),
\]
and
\[
  m_i:=-\psi_{r+1}(Y_i)\neq 0.
\]
Here $\psi_{r+1}$ stands for the differential of the same named
character. Similar notations will be used for the differentials of
representations.

For every sequence $\mathbf{N}=(N_1,N_2,\cdots N_r)$ of non-negative
integers, write
\[
  {\tbt}^{(\mathbf N)}:=\prod_{i=1}^r (t_i/t_{i+1})^{N_i},
  \quad \tbt=(t_1,t_2,\cdots,t_{r+1})\in (\K^\times_+)^{r+1}.
\]
Also write
\[
  Y^{\mathbf N}=Y_1^{N_1} Y_2^{N_2}\cdots Y_r^{N_r},
\]
which is an element in the universal enveloping algebra of the Lie
algebra of $\GL_{r+1}$.

Then
\begin{eqnarray*}
  \Lambda(\rho(a_{\tbt})\rho(Y^{\mathbf N})u)&=&\Lambda(\rho(\Ad_{a_{\tbt}}Y^{\mathbf N})\rho(a_{\tbt})u)\\
  &=&{\tbt}^{(\mathbf N)} \Lambda(\rho(Y^{\mathbf N})\rho(a_{\tbt})u)\\
  &=&\left(\prod_{i=1}^r m_i^{N_i} \right)\tbt^{(\mathbf N)}
  \Lambda(\rho(a_{\tbt})u).
\end{eqnarray*}
Therefore
\begin{equation}\label{whitt1}
  \tbt^{(\mathbf N)}\abs{\Lambda(\rho(a_{\tbt})u)}_{\pi _0}\leq \abs{m}^{-\mathbf N}||a_{\tbt} ||^{c_\rho} \abs{\rho(Y^{\mathbf
  N})u}_{\rho,0},
\end{equation}
where
\[
   \abs{m}^{-\mathbf N}:=\prod_{i=1}^r \abs{m_i}^{-N_i}.
\]

Given the positive integer $N$, write
\[
  \xi(\tbt)^N=\sum_{\mathbf N} a_{\mathbf N} \tbt^{(\mathbf
  N)},
\]
where $a_{\mathbf N}$'s are nonnegative integers.  In view of
\eqref{whitt1}, we finish the proof by setting
\[
  \abs{u}_{\rho,N}:=\sum_{\mathbf N} a_{\mathbf N}\abs{m}^{-\mathbf N}
  \abs{\rho(Y^{\mathbf N})u}_{\rho,0}.
\]

\end{proof}

\subsection{Convergence of an integral}

\begin{leml}\label{cJ}
For any non-negative continuous semialgebraic function $J$ on
$(\K^\times_+)^{r+1}$, there is a positive number $c_J$ with the
following property: for every $s\in \R^{d_\K}$ with $s>c_J$, there
is a positive integer $N$ such that
\begin{equation}\label{intj}
  \int_{(\K^\times_+)^{r+1}} (t_1 t_2\cdot\, \cdots\, \cdot
  t_r t_{r+1}')^s\, \xi(t_1, t_2, \cdots,t_r,t_{r+1}')^{-N} \, J(\tbt)\,d^\times
  {\tbt}<\infty,
\end{equation}
where
\[
   \tbt=(t_1, t_2, \cdots,t_r,t_{r+1}),\quad\quad t_{r+1}'=2(t_{r+1}^{-2} +
   t_{r+1}^{2\tau}+2)^{-\frac{1}{2}},
\]
and $\xi$ is defined in \eqref{xi}.
\end{leml}
\begin{proof}To ease the notation, we again assume that $d_\K=1$. Note that the change of
variable
\[
  \tbt\mapsto \tilde{\alpha}:=\left(\alpha_1=\frac{t_1}{t_2},
  \,\alpha_2=\frac{t_2}{t_3}, \cdots,
  \alpha_{r-1}=\frac{t_{r-1}}{t_r}, \alpha_{r}=\frac{t_r}{t'_{r+1}},
  t_{r+1}\right)
\]
is a measure preserving Nash isomorphism from $(\K^\times_+)^{r+1}$
onto itself. So $J$ is also a continuous semialgebraic function of
$\tilde \alpha$. It is well known that every continuous
semialgebraic function (on a closed semialgebraic subset of a finite
dimensional real vector space) is of polynomial growth
(\cite[Proposition 2.6.2]{BCR}). (This is the reason that we work in
the semialgebraic setting in this article.) Therefore there is a
positive number $c_J'$ such that
\[
  J(\tilde{\alpha})\leq
  \left(\prod_{j=1}^r(\alpha_j+\alpha_j^{-1})^{c_J'}\right)\times
  (t_{r+1}+t_{r+1}^{-1})^{c_J'},\quad \tbt\in (\K^\times_+)^{r+1}.
\]

Take a positive number $c_J$, large enough so that
\[
  \int_{\K^\times_+} {t_{r+1}'}^{(r+1)c_J} (t_{r+1}+t_{r+1}^{-1})^{c_J'}\, d^\times t_{r+1}<\infty.
\]
and
\[
   \int_0^1
    \alpha_j^{j c_J}(\alpha_j+\alpha_j^{-1})^{c_J'}\,d^\times
   \alpha_j<\infty, \quad j=1,2,\cdots, r.
\]
The integral (\ref{intj}) is equal to
\begin{eqnarray*}
  &&\quad  \int_{(\K^\times_+)^{r+1}} \alpha_1^s \alpha_2^{2s}\cdots \alpha_r^{rs}  {t'}_{r+1}^{(r+1)s} \prod_{j=1}^r(1+\alpha_j)^{-N} \,
   J(\tilde{\alpha})\,d^\times\tilde{\alpha}\\
   &&\leq \left(\prod_{j=1}^r \int_{\K^\times_+}
   \frac{\alpha_j^{js}(\alpha_j+\alpha_j^{-1})^{c_J'}}{(1+\alpha_j)^{N}} \,d^\times
   \alpha_j\right)\times \int_{\K^\times_+} {t_{r+1}'}^{(r+1)s} (t_{r+1}+t_{r+1}^{-1})^{c_J'}\, d^\times
   t_{r+1}.
\end{eqnarray*}

Now it is clear that the above integral converges when $s>c_J$ and
$N$ is large enough so that
\[
  \int_1^\infty  \frac{\alpha_j^{js}(\alpha_j+\alpha_j^{-1})^{c_J'}}{(1+\alpha_j)^{N}} \,d^\times
  \alpha_j<\infty,\quad j=1,2,\cdots, r.
\]

\end{proof}

\subsection{End of proof of Proposition \ref{converge}}

Take a continuous seminorm $\abs{\cdot}_{\pi,0}$ on $\pi$ and a
continuous seminorm $\abs{\cdot}_{\pi _0}$ on $\pi_0$ such that
\begin{equation}\label{mod1}
  \abs{\mu(u,v)}\leq \abs{u}_{\pi,0}\cdot \abs{v}_{\pi _0}, \quad u\in \pi,
  \,v\in\pi_0.
\end{equation}
Take a positive integer $c_\pi$ and a continuous seminorm
$\abs{\cdot}_{\pi,1}$ on $\pi$ such that
\begin{equation}\label{mod2}
   \abs{\pi(g)u}_{\pi,0}\leq ||g||^{c_\pi} \abs{u}_{\pi,1}, \quad g\in G, \, u\in \pi.
\end{equation}
Take a positive integer $c_\rho$ and a continuous seminorm
$\abs{\cdot}_{\rho,0}$ on $\rho$ such that
\begin{equation}\label{mod3}
    \abs{\Lambda(\rho(g)(u))}_{\pi _0}\leq ||g||^{c_\rho} \abs{u}_{\rho,0}, \quad g\in G_0\times \GL_{r+1}, \, u\in \rho.
\end{equation}

Now assume that
\[
  g=a_\bt g_t k,\quad  \bt=(t_1,t_2,\cdots, t_r)\in (\K^\times_+)^r, \,t\in \K^\times_+,\,
  k\in K.
\]
Write
\[
  g_t=c_{t''} n_t b_{t'}  k_t
\]
as in (\ref{decgt}). Then for all $s\in \C^{d_\K}$, $u\in \pi$,
$f\in \pi'_s$, we have
\begin{eqnarray*}
 \mu(\pi(g)u,\,\Lambda(f(g)))
 &&=\mu(\pi(c_{t''} n_t b_{t'}  a_\bt  k_t k)u,\,\Lambda(f(c_{t''} n_t b_{t'}  a_\bt  k_t k)))\\
 && \qquad (\textrm{$k_t\in \oU(E_3)$ commutes with $a_\bt\in \GL_r$})\\
 &&=\mu(\pi(n_t b_{t'}  a_\bt  k_t k)u,\,\Lambda(f(n_t b_{t'}  a_\bt  k_t k)))\\
 &&\qquad (\textrm{By Lemma \ref{srinv} and the fact that
 $c_{t''}\in G_0\subset S_r$})\\
 &&=\mu(\pi(n_t b_{t'}  a_\bt  k_t k)u,\,(t_1 t_2\cdots t_r t')^s \Lambda(\rho(a_\bt b_{t'}) f(k_t k)))\\
 &&\qquad (\textrm{$n_t\in N_{P_{r+1}'}$ and $a_\bt b_{t'}\in
 \GL_{r+1}$}).
\end{eqnarray*}
Therefore by \eqref{mod1}, \eqref{mod2} and the fact that the norm
function $||\cdot ||$ on $G'$ is right $K'$-invariant, we have
\begin{eqnarray}
 \label{nint}&&\quad \abs{\mu(\pi(g)u,\,\Lambda(f(g)))}\\
  \nonumber &&\leq (t_1 t_2\cdots t_r t')^{\Re(s)} \times
  ||n_t b_{t'}  a_\bt||^{c_\pi} \times \abs{u}_{\pi,1} \times
  \abs{\Lambda(\rho(a_\bt b_{t'}) f(k_t k))}_{\pi _0}.
\end{eqnarray}

Let $J$ be the (nonnegative continuous semialgebraic) function on
$(\K^\times_+)^{r+1}$ defined by
\[
  J(\bt,t):= ||n_t b_{t'}  a_\bt||^{c_\pi}\times || a_\bt
  b_{t'}||^{c_\rho}\times \delta_{P_r}^{-1}(a_{\bt})\, \delta_{B_r}^{-1}(a_{\bt})\,
  J_{E_0'}(t).
\]
Let $c_\mu:=c_J$ be as in Lemma \ref{cJ}, and assume that the real
part $\Re(s)>c_\mu$. Let $N$ be a large integer as in Lemma \ref{cJ}
so that
\[
  c_{s,N}:=\int_{(\K^\times_+)^{r}\times  \K^\times_+} (t_1 t_2\cdots t_r t')^{\Re(s)}\xi(\bt,t')^{-N} J(\bt,t) \,d^\times \bt\, d^\times
   t<\infty.
\]

Take a continuous seminorm $\abs{\cdot}_{\rho,N}$ on $\rho$ as in
Lemma \ref{whittaker}. Then we have
\begin{equation}\label{nint2}
  \abs{\Lambda(\rho(a_\bt b_{t'}) f(k_t k))}_{\pi _0}\leq \xi(\bt,t')^{-N}||a_\bt
  b_{t'}||^{c_\rho} \abs{ f(k_t k)}_{\rho,N}\leq  \xi(\bt,t')^{-N}||a_\bt
  b_{t'}||^{c_\rho} \abs{ f}_{\pi'},
\end{equation}
where
\[
  \abs{ f}_{\pi'}:=\max\{\abs{ f(k')}_{\rho,N}\mid k'\in K'\},
\]
which defines a continuous seminorm on $\pi_s'$. Combining
(\ref{nint}) and (\ref{nint2}), we get
\begin{eqnarray}
 \label{nint3}&&\quad \abs{\mu(\pi(g)u,\,\Lambda(f(g)))}\\
  \nonumber &&\leq \abs{u}_{\pi,1}\times \abs{ f}_{\pi'}\times (t_1 t_2\cdots t_r t')^{\Re(s)} \times \xi(\bt,t')^{-N}\times ||n_t b_{t'}  a_\bt||^{c_\pi} \times ||a_\bt
  b_{t'}||^{c_\rho}.
\end{eqnarray}

Then by Proposition \ref{intf} and (\ref{nint3}), we have
\begin{eqnarray*}
 &&\quad \int_{S_r\backslash G}\abs{\mu(\pi(g)u,\,\Lambda(f(g)))}\,dg\\
 &&=\int_{(\K^\times_+)^r \times \K^\times_+\times K}
  \abs{\mu(\pi(a_\bt g_t k)u,\,\Lambda(f(a_\bt g_t k)))}\,\delta_{P_r}^{-1}(a_{\bt})\, \delta_{B_r}^{-1}(a_{\bt})\, J_{E_0'}(t)\,d^\times \bt \,d^\times
  t\,dk\\
   &&\leq \abs{u}_{\pi,1}\times \abs{ f}_{\pi'}\times \int_{(\K^\times_+)^{r}\times  \K^\times_+} (t_1 t_2\cdots t_r t')^{\Re(s)}\xi(\bt,t')^{-N} J(\bt,t) \,d^\times \bt\, d^\times
   t\\
   &&=c_{s,N}\times  \abs{u}_{\pi,1}\times \abs{ f}_{\pi'}.
    \end{eqnarray*}
Therefore the integral $I_\mu(f,u)$ converges absolutely. Finally,
\[
   \abs{I_\mu(f,u)}\leq \int_{S_r\backslash
   G}\abs{\mu(\pi(g)u,\,\Lambda(f(g)))}\,dg\leq c_{s,N}\times  \abs{u}_{\pi,1}\times \abs{
   f}_{\pi'},
\]
which proves the continuity of $I_\mu$. This finishes the proof of
Proposition \ref{converge}.

\end{document}